\documentclass[11pt,a4paper,reqno,psamsfonts]{amsart}
\usepackage[latin1]{inputenc}
\usepackage{amsmath}
\usepackage{amsfonts}
\usepackage{amssymb}
\usepackage{amsthm}
\usepackage{mathrsfs}
\usepackage{graphicx}
\usepackage{empheq}
\usepackage{bbm}
\usepackage{dsfont}

        \def\1{1\kern-3pt\rm l}

\usepackage{setspace}

\newcommand{\real}{\mathbb{R}}
\newcommand{\complex}{\mathbb{C}}
\newcommand{\nat}{\mathbb{N}}

\newcommand{\xtnd}{\widetilde}

\newcommand{\dd}{\textrm{\rm{d}}}
\newcommand{\ddc}{\textrm{\rm{dd}}^{c}}

\newtheorem{thm}{Theorem}[section]
\newtheorem{cor}[thm]{Corollary}

\newtheorem{lem}[thm]{Lemma}

\theoremstyle{remark}

\newtheorem{exmp}[thm]{Example}

\numberwithin{equation}{section}

\begin{document}
\baselineskip15pt

\title[Some achieved wedge products of positive currents]{Some achieved wedge products of positive currents}
\author{Ahmad K. Al Abdulaali}
\address{Department of Mathematics and Statistics, College of Science, King Faisal University, P.O  Box 380 Al-Ahsa, 31982, Saudi Arabia}
\email{aalabdulaaly@kfu.edu.sa}

\begin{abstract}
In this paper, we study the existence of the current $g T$ for  positive plurisubharmonic currents $T$ and unbounded plurisubharmonic functions $g$.
\end{abstract}

\maketitle

\section{Introduction}

The wedge product is a very vital subject in the field of currents, and mainly targeted indeed. In general, wedge product of currents can not be achieved unless further conditions are considered. Throughout this paper we consider $\Omega$ to be an open subset of $\complex^{n}$ and $T$ to be a current of bi-dimension $(p,p),  p\geq 1$. We denote by $Psh^{-}(\Omega)$ the set of all negative plurisubharmonic functions on $\Omega$. For a function $g \in Psh^{-}(\Omega)$, put $L_{g}$ to be the set of all locus points of $g$ which consists of the points $z \in \Omega$ where $g$ is unbounded in every neighborhood of $z$. The pole set of $g$ is by definition $P_{g}=\{ g=- \infty\}$. It is obvious that $L_{g}$ is closed, and $\overline{P_{g}} \subset L_{g}$. Recall also that $T$ is said to be closed if $\dd T=0$, and is said to be plurisubharmonic (resp. plurisuperharmonic) if $\ddc T\geq 0$ (resp. $\ddc T\leq 0$). 

Our main concern is to give a definition of the current $gT$. Obviously, apart of $L_{g}$, the current $gT$ is well defined. The real challenge is studying the existence of this product across $L_g$. In such a situation, this product may have no sense due to the behaviors of $g$ and $T$. For example, In $\complex$ put
\begin{equation}\label{gT:exmp}
T=\frac{(-\log{|z|^{2}})^{\frac{-3}{2}}}{|z|^{2}} \ \textrm{and} \  g=log|z|^{2}.
\end{equation}
Then we get a positive and  integrable current $T$ of bi-dimension $(1,1)$ on $B(0,1)$ which is plurisubharmonic outside the origin, and a function $g\in Psh^{-}(B(0,1)) \cap \mathcal{C}^{\infty}(B(0,1)\setminus \{ 0\})$. Despite the fact that $\mathcal{H}_{2(1)-2}(\{ 0\})=1$, the current $gT$ is of infinite mass near the origin. Now we can feel the motivation behind the paper which is basically about finding sufficient conditions on $T$ and $L_{g}$ that make $gT$ well defined. The study is also consistent with the evolution of the subject as the case when $T$ is closed was considered before in many works. In fact, Demailly \cite{De2} (1993) proved the existence of $gT$ and $\ddc g\wedge T$ as soon as $\mathcal{H}_{2p-1}(L_{g} \cap \rm{Supp} \ T)=0$.  Forn\ae ss and Sibony \cite{Fo-Si} (1994) generalized the work of Demailly to higher Hausdorff dimension when they succeeded to define the currents $gT$ and $\ddc g \wedge T$ where $\mathcal{H}_{2p}(L_{g})=0$. In both studies the closedness property played a main role since it was used to give the relation between $gT$ and $\ddc g \wedge T$. Namely, with such a property, the definition of $\ddc g \wedge T$ is given by $\ddc (gT)$ in the sense of distribution. Unfortunately, this relation becomes more complicated once we deal with $\ddc$-signed  currents as the terms $\dd g \wedge \dd^{c}T$ and $g \ddc T$ have their contribution. This fact causes difficulties to achieve a definition of $gT$, and forced some interested researchers to study the current $\ddc g \wedge T$ separately. In what follows we summarize our main results.

\medskip

\emph{Let $T$ be a positive  plurisubharmonic current of bi-dimension $(p,p)$ on  $\Omega $ and $g \in Psh^{-}(\Omega)\cap \mathcal{C}^{1}(\Omega\setminus L_{g})$. Then  in each of the following cases the current $gT$ is well defined, that means if  $(g_{j})$ is a sequence of decreasing smooth plurisubharmonic functions on $\Omega$ converging to $g$ in $\mathcal{C}^{1}(\Omega \setminus L_{g})$, then $g_{j}T$ converges weakly$^{*}$ to a current denoted by $gT$.
\begin{enumerate}
\item  $L_{g}$ is compact, $L_{g}=P_{g}$ and $p\geq 2$. (\textbf{Theorem \ref{p geq 2}})
\item $L_{g}$ is compact and $g \dd^{c} T$ is well defined. (\textbf{Theorem \ref{1st:main}})
\item $\mathcal{H}_{2p-1}(L_{g}\cap \rm{Supp} \ T)=0$ and $g \dd^{c}T$ is well defined. (\textbf{Theorem \ref{A:H2p-1}})
\item $\mathcal{H}_{2p-2}(L_{g}\cap \rm{Supp} \ T)$ is locally finite.(\textbf{Theorem \ref{A:H2p-2}})
\end{enumerate}
}

\medskip

The precautions taken in the previous results on the thickness of $L_{g}$ and the properties of $T$ are extremely important to guaranty the existence of $gT$. Actually, the failure to define $gT$ with the choices of (\ref{gT:exmp}) due to fact that both $g\dd^{c}T$ and $g \ddc T$ are of infinite mass near the origin. In a recent work, Al Abdulaali-El Mir \cite{Ah-El3} obtained the current $g^{k}T, \ k>0$ when $g$ is radial and $L_g$ is reduced to a single point.

The second part of the paper is devoted to discuss the current $\ddc g \wedge T$. As a consequence of the discussions of this part, we show that the quantity
\begin{equation}
\mu(T,g)=\lim_{r \to -\infty} \int_{ \{ g<r\} } T \wedge (\ddc g) \wedge \beta^{p-1}, \ \beta=\ddc \| z \|^{2}
\end{equation}
exists under the hypotheses of Theorem \ref{1st:main}. Furthermore, by a counterexample it is shown that the induced results can not be obtained for the case of positive plurisuperharmonic currents without further hypotheses.

\section{The Current $gT$}

Let us start with a result due to Al Abdulaali \cite{Ah6}. In the following few lines we include the proof in our settings.

\begin{lem}\label{compact}
Let $T$ be a positive plurisubharmonic current of bi-dimension $(p,p)$ on $\Omega $ and $g \in Psh^{-}(\Omega)\cap \mathcal{C}^{1}(\Omega\setminus A)$ for some compact subset $A$ of $\Omega$. Assume that $(g_{j})$ is a sequence of decreasing smooth plurisubharmonic functions on $\Omega$ converging to $g$ in $\mathcal{C}^{1}(\Omega \setminus A)$. Then
\begin{enumerate}
\item $g \ddc T$ is a well defined current on $\Omega$, and the trivial extension $\xtnd{\ddc g \wedge T}$ exists.
\item $\ddc g \wedge T$ is a well defined current as soon as  $A$, in addition,  is complete pluripolar and $p\geq 2$.
\item $\ddc g \wedge T$ is a well defined current when $A$ is considered to be a single point.
\end{enumerate}
\end{lem}

\begin{proof}
Let $W$ and $W^{'}$ be neighborhoods of $A$ such that $W \Subset W^{'} \Subset \Omega$, and take a positive function $f \in \mathcal{C}^{\infty}_{0}(W^{'})$ so that $f=1$ on a neighborhood of $W$. Then we have
\begin{equation}\label{ddcfg}
 \begin{split}
 \int_{W^{'}} \ddc (f g_{j})\wedge T\wedge \beta^{p-1}= \int_{W^{'}} f g_{j} \ddc T \wedge \beta^{p-1}\leq 0.
 \end{split}
\end{equation}
This implies that 

\begin{equation}\label{fddcg}
 \begin{split}
0 & \leq \int_{W^{'}}f \ddc  g_{j}\wedge T \wedge \beta^{p-1} -  \int_{W^{'}} f g_{j} \ \ddc T \wedge \beta^{p-1} \\ & \leq \left| \int_{W^{'}}\dd  g_{j} \wedge \dd^{c}f \wedge T \wedge \beta^{p-1}\right|+ \left| \int_{W^{'}}\dd  f \wedge \dd^{c}g_{j} \wedge T \wedge \beta^{p-1} \right| 
\\ &  \ + \left| \int_{W^{'}} g_{j} \ddc f \wedge T \wedge \beta^{p-1}\right|. 
 \end{split}
\end{equation}
Thanks to the properties of $f$, each term of the first line integrals of (\ref{fddcg}) is uniformly bounded. Therefore, one can infer the existence of both extensions $\xtnd{g\ddc T}$ and $\xtnd{\ddc g \wedge T}$. Notice that, the current $g \ddc T$ is well defined by the monotone convergence. And by Banach-Alaoglu the sequence $(\ddc g_{j}\wedge T)$ has a subsequence  $(\ddc g_{j_{s}}\wedge T)$ which converges weakly$^{*}$ to a current denoted by $S$. To show (2), we first note that $\ddc S$ is a well defined current as well. Hence by \cite{Da-Kh-El} the residual current $R=\xtnd{\ddc S}- \ddc (\xtnd{\ddc g \wedge T})$ is positive and supported in $A$. Now, if we set $F:= S - \xtnd{\ddc g \wedge T}$ we find clearly that $F$ is a positive current where $$ \ddc F =  \ddc S - \ddc (\xtnd{\ddc g \wedge T}) \geq \ddc S-  \xtnd{\ddc S} \geq 0.$$ As $F$ is a compactly supported current with bi-dimension $(p-1,p-1)$, one can deduce that $F \equiv 0$. The third statement comes immediately from the fact that the distribution $\mu:= (S - \xtnd { \ddc g \wedge T})\wedge \beta^{p-1}$ is positive and supported in $A$. Indeed, $A$ can be assumed to be the origin, and hence there exists a positive constant $c$ such that $\mu=c\delta_{0}$ where $\delta_{0}$ is the Dirac measure. Clearly, the constant $c$ is independent from the choice of $j_{s}$ since 
\begin{equation*}
\begin{split}
c=\mu(f)&=\lim_{j_{s} \to \infty} \int_{W^{'}} f  \ddc g_{j_{s}} \wedge T \wedge \beta^{p-1} - \int_{W^{'}} f \xtnd{ \ddc g \wedge T} \wedge \beta^{p-1}\\
&= \int_{W^{'}}( g  \ddc f +2 \dd g \wedge \dd^{c} f ) \wedge T \wedge \beta^{p-1} \\
& \ \ + \int_{W^{'}} f   g \ddc T \wedge \beta^{p-1} - \int_{W^{'}} f \xtnd{ \ddc g \wedge T} \wedge \beta^{p-1}.
\end{split}
\end{equation*}
In other words, $\ddc g \wedge T$ is well defined.
\end{proof}

As a consequence of the previous result, the wedge products in \cite{Al-Ba2} and \cite{Ah4} can be generalized to case when $\mathcal{H}_{2p-2}(A)$ is locally finite. (See \cite{Ah6})

\begin{lem}\label{dg:dcg}
Let $T$ be a positive plurisubharmonic current of bi-dimension $(p,p), \ p \geq 1$ on  $\Omega $ and $g \in Psh^{-}(\Omega)\cap \mathcal{C}^{1}(\Omega\setminus L_{g})$. If $L_{g}$ is compact, then for every compact subset $K$ of $\Omega$ we have $\left\Vert \displaystyle{\frac{\dd g \wedge \dd^{c}g}{(-g)^{1+\varepsilon}}}\wedge T\right\Vert_{K \setminus L_{g}} < \infty$, $0< \varepsilon <1$.
\end{lem}

\begin{proof}
Notice first that for every $0 < \varepsilon <1$ the function $-(-g)^{1-\varepsilon}$ is plurisubharmonic. Hence, by Lemma \ref{compact} the current $-\ddc (-g)^{1-\varepsilon} \wedge T$ is of locally finite mass across $L_{g}$. But, by simple computation one has 
\begin{equation}\label{dg:dcg:1}
-\ddc (-g)^{1-\varepsilon}=\frac{\dd g \wedge \dd^{c}g}{(-g)^{1+\varepsilon}}+\frac{\ddc g}{(-g)^{\varepsilon}}.
\end{equation}
This shows the result.
\end{proof}

\begin{thm}\label{p geq 2}
 Let $T$ be a positive plurisubharmonic current of bi-dimension $(p,p), \ p \geq 1$ on  $\Omega $ and $g \in Psh^{-}(\Omega)\cap \mathcal{C}^{1}(\Omega\setminus L_{g})$. If $L_{g}$ is compact, then for all $0< \varepsilon <1$ the current $\vert g \vert^{1-\varepsilon}T$ is well defined. Moreover, if $L_g=P_g$ and $p \geq 2$, then $gT$ is well defined.
\end{thm}

\begin{proof}
Take $W$, $W^{'}$, $f$ and $g_{j}$ as in the proof of Lemma \ref{compact}, and for $0<\varepsilon<1$ set $u_{j}=-(-g_{j})^{1-\varepsilon}$, $j \in \nat$. Observe that $(u_{j})_{j}$ is a sequence of negative plurisubharmonic functions where 
\begin{equation}
\dd u_{j}= (1-\varepsilon)\frac{\dd g_{j}}{(-g_{j})^{\varepsilon}}.
\end{equation}
Clearly we have
\begin{equation}
\int_{W^{'}} \ddc (-u_{j}f |z|^{2}) \wedge T \wedge \beta^{p-1}=\int_{W^{'}}  -u_{j}f |z|^{2} \ddc T \wedge \beta^{p-1}
\end{equation} 
Therefore,
\begin{equation}\label{betap}
\begin{split}
\mathcal{I}_{j}:= \int_{W^{'}} -u_{j}f T \wedge \beta^{p}& \leq \int_{W^{'}}  -u_{j}f |z|^{2} \ddc T \wedge \beta^{p-1}\\
& \ + \int_{W^{'}} f |z|^{2}  \ddc u_{j} \wedge T \wedge \beta^{p-1} \\
&  \ +2\left\vert \int_{W^{'}} f \dd u_{j} \wedge \dd^{c}|z|^{2} \wedge  T \wedge \beta^{p-1} \right\vert + \vert \mathcal{O}^{'}(f)\vert ,
\end{split}
\end{equation} 
where $\mathcal{O}^{'}(f)$ consists of all terms involving $\dd f, \ \dd^{c}f$ and $\ddc f$. The first two line integrals of the right hand side of (\ref{betap}) are uniformly bounded, thanks to Lemma \ref{compact}. Furthermore, Cauchy-Schwartz inequality shows that 

\begin{equation}\label{uj:T}
\begin{split}
&\left\vert \int_{W^{'}} f \dd u_{j}  \wedge \dd^{c}|z|^{2}\wedge  T \wedge \beta^{p-1} \right\vert \\
& \leq \left\vert \int_{W^{'}} \frac{f}{-\delta u_{j}} \dd u_{j} \wedge \dd^{c} u_{j} \wedge T \wedge \beta^{p-1} \right\vert^{\frac{1}{2}} \times  \left\vert \int_{W^{'}}  (-\delta u_{j}) f \dd |z|^{2} \wedge \dd^{c} |z|^{2}\wedge  T \wedge \beta^{p-1} \right\vert^{\frac{1}{2}}\\
& \leq \left\vert \int_{W^{'}} \frac{f}{\delta (-g_{j})^{1+\varepsilon}} \dd g_{j} \wedge \dd^{c} g_{j} \wedge T \wedge \beta^{p-1} \right\vert^{\frac{1}{2}} \times \left\vert \int_{W^{'}}  -f u_{j}   T \wedge \beta^{p} \right\vert^{\frac{1}{2}},
\end{split}
\end{equation}
where $\delta$ is a positive constant chosen so that $\delta \dd |z|^{2} \wedge \dd^{c} |z|^{2} \leq \ddc |z|^{2}$. By complying the last two inequalities, we have
\begin{equation}
\mathcal{I}_{j}\leq \mathcal{M}_{j}+ \mathcal{A}_{j} \times \mathcal{I}^{\frac{1}{2}}_{j}.
\end{equation}
And as $\mathcal{M}_{j}$ together with $\mathcal{A}_{j}$ are uniformly bounded, one can conclude the definition of $|g|^{1-\varepsilon}T$. Suppose now that $p\geq 2$ and $L_g=P_g$. The current $-T \wedge \ddc (-g)^{1-\varepsilon}$ is positive and plurisubharmonic on $\Omega$. Hence, the precedent part guarantees the existence of the trivial extension of  $-|g|^{1-\varepsilon} T \wedge \ddc (-g)^{1-\varepsilon}$, and obviously the current $\displaystyle{\frac{\dd g \wedge \dd^{c}g}{(-g)^{2\varepsilon}}\wedge T}$ is of locally finite mass. Notice also that 
\begin{equation}
\begin{split}
&\lim_{j \to \infty}\left\vert \int_{W^{'}} f \dd g_{j}  \wedge \dd^{c}|z|^{2}\wedge  T \wedge \beta^{p-1} \right\vert \\
& \leq \lim_{j \to \infty} \left\vert \int_{W^{'}} \frac{f}{\delta(- g_{j})^{\varepsilon}} \dd g_{j} \wedge \dd^{c} g_{j} \wedge T \wedge \beta^{p-1} \right\vert^{\frac{1}{2}} \times  \left\vert \int_{W^{'}}  \delta f(-g_{j})^{\varepsilon}   T \wedge \beta^{p} \right\vert^{\frac{1}{2}}\\
& < \infty.
\end{split}
\end{equation}
Therefore, by similar argument as above, one can replace $u_{j}$ by $g_{j}$ in (\ref{betap}) and deduce that $g_{j}T$ converges to a current denoted by $gT$. 
\end{proof}

\begin{thm}\label{g:1+eps}
 Let $T$ be a positive pluriharmonic current of bi-dimension $(p,p), \ p \geq 2$ on  $\Omega $ and $g \in Psh^{-}(\Omega)\cap \mathcal{C}^{1}(\Omega\setminus L_{g})$. If $L_{g}$ is compact and $L_g=P_g$, then for all $0< \varepsilon <1$ the current $\vert g \vert^{1+\varepsilon}T$ is well defined.
\end{thm}

\begin{proof}
In virtue of the precedent argument, the currents $(-g)^{\varepsilon}\ddc g \wedge T$ and $\displaystyle{\frac{\dd g \wedge \dd^{c}g}{(-g)^{1-\varepsilon}}\wedge T}$ are of locally finite mass across $L_g$. On the other hand
\begin{equation}
\ddc (-g)^{1+\varepsilon}\wedge T=\varepsilon (1+\varepsilon) \frac{\dd g \wedge \dd^{c}g}{(-g)^{1-\varepsilon}} \wedge T- (1+\varepsilon) (-g)^{\varepsilon}\ddc g \wedge T.
\end{equation}
This implies that the trivial extension $\xtnd{\ddc (-g)^{1+\varepsilon} \wedge T}$ exists. Set $v_{j}=(-g_{j})^{1+\varepsilon}$, $j \in \nat$. By the features of $T$ we have 
\begin{equation}
\int_{W^{'}} \ddc (v_{j}f |z|^{2}) T \wedge \beta^{p-1}=0
\end{equation} 
Now, by analogous discussion as in the previous proof, we infer the definition of $(-g)^{1+\varepsilon} T$ since
\begin{equation}
\begin{split}
 \int_{W^{'}}& v_{j}f T \wedge \beta^{p}\\
& \leq \left\vert \int_{W^{'}} f |z|^{2}  \ddc v_{j} \wedge T \wedge \beta^{p-1} \right\vert + \vert \mathcal{O}^{'}(f)\vert\\
&  \ +2 \left\vert \int_{W^{'}} \frac{f}{\delta v_{j}} \dd v_{j} \wedge \dd^{c} v_{j} \wedge T \wedge \beta^{p-1} \right\vert^{\frac{1}{2}} \times \left\vert \int_{W^{'}}v_{j}  f    T \wedge \beta^{p} \right\vert^{\frac{1}{2}}.
\end{split}
\end{equation} 
\end{proof}

\begin{lem}\label{single}
 Let $T$ be a positive plurisubharmonic current of bi-dimension $(p,p), \ p \geq 1$ on  $\Omega $ and $g \in Psh^{-}(\Omega)\cap \mathcal{C}^{1}(\Omega\setminus L_{g})$. If $L_{g}$ is a single point, then $gT$ is well defined.
\end{lem}

\begin{proof}
Without loss of generality, one can assume that $\Omega$ is the unit ball and $L_g$ is the origin. Take $\chi \in \mathcal{C}^{\infty}_{0}(B(0,\frac{1}{2}))$ so that $\chi=1$ on  a neighborhood of $B(0,\frac{1}{4})$. First notice that for all $0<t<1$ we have
\begin{equation}\label{chi:gj}
\lim_{j \to \infty} \left\vert \int_{\{ |z| \leq t\}} \dd (- \chi g_{j}) \wedge \dd^{c}T \wedge \beta^{p-1}\right\vert = - \int_{\{ |z| \leq t\}} \chi g \ddc T \wedge \beta^{p-1}.
\end{equation}
But 
\begin{equation}
\begin{split}
&\lim_{j \to \infty} \left\vert \int_{\{ |z| \leq \frac{1}{2}\}} - \chi g_{j} \dd |z|^{2} \wedge \dd^{c}T \wedge \beta^{p-1} \right\vert\\ &=\lim_{j \to \infty}  \left\vert \int_{0}^{\frac{1}{2}}\dd t \int_{\{ |z|=t \}} - \chi g_{j} \dd^{c}T \wedge \beta^{p-1} \right\vert\\
&=\lim_{j \to \infty}  \left\vert \int_{0}^{\frac{1}{2}}\dd t \int_{\{ |z|\leq t \}} [\dd (- \chi g_{j}) \wedge \dd^{c}T- \chi g_{j}  \ddc T] \wedge \beta^{p-1} \right\vert< \infty.
\end{split}
\end{equation}
Therefore, 
\begin{equation}
\sup_{j}\left\vert  \int_{\{ |z| \leq \frac{1}{2}\}} -g_{j} \dd (\chi|z|^{2}) \wedge \dd^{c} T \wedge \beta^{p-1} \right\vert < \infty.
\end{equation}
By a simple computation, one finds that
\begin{equation}
\chi |z|^{2}\dd g_{j}\wedge \dd^{c}T=  g_{j} \dd (\chi|z|^{2}) \wedge \dd^{c} T +  g_{j} \chi|z|^{2} \ddc T.
\end{equation}
Hence, 
\begin{equation}\label{chi:dg}
\sup_{j}\left\vert  \int_{\{ |z| \leq \frac{1}{2}\}}-\chi |z|^{2}\dd g_{j}\wedge \dd^{c} T \wedge \beta^{p-1} \right\vert < \infty.
\end{equation}
Now, once again Stokes' formula shows that
\begin{equation}\label{ddc(gT1)}
\int_{\{ |z| \leq \frac{1}{2}\}} -\ddc (\chi \vert z \vert^{2}) \wedge g_{j} T  \wedge \beta^{p-1}= \int_{\{ |z| \leq \frac{1}{2}\}} - \chi \vert z \vert^{2}  \ddc (g_{j}T)\wedge \beta^{p-1}.
\end{equation}
But
\begin{equation}
\begin{split}
\int_{\{ |z| \leq \frac{1}{2}\}} -\chi  \vert z \vert^{2} \ddc (g_{j}T)\wedge \beta^{p-1}&= \int_{\{ |z| \leq \frac{1}{2}\}} -\chi  \vert z \vert^{2}  \ddc g_{j}\wedge T\wedge \beta^{p-1} \\ & \ \ + \int_{\{ |z| \leq \frac{1}{2}\}} -\chi  \vert z \vert^{2} g_{j} \ddc T\wedge \beta^{p-1} \\
& \ \ + 2 \int_{\{ |z| \leq \frac{1}{2}\}} -\chi  \vert z \vert^{2}  \dd g_{j}\wedge \dd^{c} T\wedge \beta^{p-1}.
\end{split}
\end{equation}
Thus, we infer that
\begin{equation}
\begin{split}
\sup_{j} \left\vert \int_{\{ |z| \leq \frac{1}{2}\}} -\chi  \vert z \vert^{2} \ddc (g_{j}T)\wedge \beta^{p-1} \right\vert < \infty, 
\end{split}
\end{equation}
thanks to Lemma \ref{compact} and (\ref{chi:dg}). On the other hand, the left hand side of (\ref{ddc(gT1)}) involves $\mathcal{O}(\chi)$ which consists of the terms where $\dd \chi$,  $\dd^{c}\chi$ and $\ddc \chi$ appear. Notice that $\mathcal{O}(\chi)$ is under control since the origin is isolated. Hence,   
\begin{equation}
\begin{split}
- \int_{\{ |z| \leq \frac{1}{4}\}} g_{j}  T  \wedge \beta^{p} &\leq - \int_{\{ |z| \leq \frac{1}{2}\}} \chi  g_{j}  T  \wedge \beta^{p} \\ & \leq \vert \mathcal{O}(\chi ) \vert + \left\vert  \int_{\{ |z| \leq \frac{1}{2}\}} -\chi  \vert z \vert^{2} \ddc (g_{j}T) \wedge \beta^{p-1} \right\vert < \infty.
\end{split}
\end{equation} 
Clearly, the current $gT$ is obtained by the monotone convergence of $g_{j}T$.
\end{proof}

\begin{thm}\label{1st:main}
Let $T$ be a positive plurisubharmonic current of bi-dimension $(p,p), \ p \geq 1$ on  $\Omega $ and $g \in Psh^{-}(\Omega)\cap \mathcal{C}^{1}(\Omega\setminus L_{g})$. If the current $g\dd^{c} T$ is well defined and $L_{g}$ is compact, then $ g T$ is a well defined current on $\Omega$.
\end{thm}
The result generalizes the case when $\dd T=0$. One can also implement it for the currents $T$ where $\dd T$ has $L^{q}$ coefficients, $q>1$. 
\begin{proof}
We keep the notation of the proof of Lemma \ref{compact} taking into consideration that $L_{g} \subset W$. It is obvious that $\dd g_{j}\wedge \dd^{c}T=\dd (g_{j}\dd^{c}T)-g_{j}\ddc T$. Hence by Lemma \ref{compact}, the current $\dd g \wedge \dd^{c}T$ is well defined. Now, by applying Stokes' formula we have
\begin{equation}\label{ddc(gT)}
\int_{W^{'}} -\ddc (f \vert z \vert^{2}) \wedge g_{j} T  \wedge \beta^{p-1}= \int_{W^{'}} - f \vert z \vert^{2}  \ddc (g_{j}T)\wedge \beta^{p-1}.
\end{equation}
But
\begin{equation}
\begin{split}
\int_{W^{'}} -f \vert z \vert^{2} \ddc (g_{j}T)\wedge \beta^{p-1}&= \int_{W^{'}} f \vert z \vert^{2}  \ddc g_{j}\wedge T\wedge \beta^{p-1} \\ & \ \ + \int_{W^{'}} f \vert z \vert^{2} g_{j} \ddc T\wedge \beta^{p-1} \\
& \ \ + 2 \int_{W^{'}} f \vert z \vert^{2}  \dd g_{j}\wedge \dd^{c} T\wedge \beta^{p-1}.
\end{split}
\end{equation}
Which means that
\begin{equation}
\begin{split}
\sup_{j} \left\vert \int_{W^{'}} -f \vert z \vert^{2} \ddc (g_{j}T)\wedge \beta^{p-1} \right\vert < \infty 
\end{split}
\end{equation}
because of Lemma \ref{compact}. Meanwhile, the left hand side of (\ref{ddc(gT)}) involves terms $\mathcal{O}(f)$ where $\dd f$,  $\dd^{c}f$ and $\ddc f$ appear. And the properties of $f$ make these terms defined and uniformly bounded as the locus points of $g$ are avoided. Hence,   
\begin{equation}
\begin{split}
- \int_{W} g_{j}  T  \wedge \beta^{p} &\leq - \int_{W^{'}} f g_{j}  T  \wedge \beta^{p} \\ & \leq \vert \mathcal{O}(f) \vert + \left\vert  \int_{W^{'}} -f \vert z \vert^{2} \ddc (g_{j}T) \wedge \beta^{p-1} \right\vert < \infty.
\end{split}
\end{equation} 
This yields to the current $gT$.
\end{proof}
Next we give conditions on the locus points of $g$ that allow the existence of $gT$ regardless the compactness property.

\begin{thm}\label{A:H2p-1}
Let $T$ be a positive  plurisubharmonic current of bi-dimension $(p,p)$ on  $\Omega $ and $g\in Psh^{-}(\Omega) \cap\mathcal{C}^{1}(\Omega \setminus L_{g})$. If the current $g\dd^{c} T$ is well defined and $\mathcal{H}_{2p-1}(L_{g}\cap \textrm{Supp} \ T)=0$, then the current $gT$ is well defined.
\end{thm} 

\begin{proof}
 Let us assume that $ 0 \in {\textrm{Supp} \ T}\cap L_{g}$. Since $\mathcal{H}_{2p-1}(L_{g}\cap \textrm{Supp} \ T)=0$, then by \cite{Bi} and \cite{Sh}, there exist a system of coordinates $(z^{\prime},z^{\prime\prime})\in \complex^{s}\times \complex^{n-s}$, $s=p-1$  and a polydisk $\bigtriangleup^{n}=\bigtriangleup^{\prime}\times \bigtriangleup^{\prime\prime}$ such that $\overline{\bigtriangleup^{\prime}}\times \partial \bigtriangleup^{\prime\prime}\cap ({\textrm{Supp} \ T}\cap L_g)=\emptyset$. Now, take $0<t<1$ so that $\bigtriangleup^{\prime}\times  \{ z^{\prime\prime}, t<\vert z^{\prime\prime}\vert<1\} \cap ({\textrm{Supp} \ T}\cap L_{g})=\emptyset$. As $\overline{ \bigtriangleup^{n}}\cap L_{g}$ is compact set, one can find a neighborhood $\omega$ of $\overline{ \bigtriangleup^{n}}\cap  L_g$ such that $\omega\cap (\bigtriangleup^{\prime}\times  \{ z^{\prime\prime}, t<\vert z^{\prime\prime}\vert<1\})=\emptyset$. Let $a\in (t,1)$. Notice that the slice $\langle T,\pi,z^{\prime}\rangle$ exists for a.e. $z^{\prime}$, and is a positive plurisubharmonic current of bi-dimension $(1,1)$ on $\Omega$, supported in $\lbrace z^{\prime}\rbrace\times \bigtriangleup^{n-p+1}$. Hence, by Theorem \ref{1st:main}, the sequence $\langle  g_{j} T, \pi, z^{\prime}\rangle$ is weakly$^{*}$ convergent since $\pi^{-1}(z^{\prime}) \cap L_{g}$ is a compact subset of $\lbrace z^{\prime}\rbrace\times \bigtriangleup^{n-p+1}$. Thus by applying the slice formula we have
\begin{equation}
\begin{split}
\lim_{j \to \infty} \int_{\bigtriangleup^{\prime}\times \bigtriangleup^{\prime\prime}} g_{j} T\wedge \pi^{*}\beta^{\prime p-1} \wedge \beta^{''} &=\lim_{j \to \infty} \int_{z^{\prime}} \langle g_{j} T, \pi,z^{\prime}\rangle \beta^{\prime p-1} \wedge \beta^{''} \\&= \int_{z^{\prime}} \langle  g T, \pi,z^{\prime}\rangle \beta^{\prime p-1} \wedge \beta^{''}\\&= \int_{\bigtriangleup^{\prime}\times \bigtriangleup^{\prime\prime}} g T\wedge \pi^{*}\beta^{\prime p-1} \wedge \beta^{''}. \nonumber
\end{split}
\end{equation}
This completes the proof.
\end{proof}

\begin{thm}\label{A:H2p-2}
Let $T$ be a positive  plurisubharmonic current of bi-dimension $(p,p)$ on  $\Omega $ and $g\in Psh^{-}(\Omega) \cap\mathcal{C}^{1}(\Omega \setminus L_{g})$. If $\mathcal{H}_{2p-2}(L_{g}\cap \textrm{Supp} \ T)$ is locally finite, then the current $gT$ is well defined.
\end{thm} 
\begin{proof}
For each $z^{\prime}$ we set $L_{g}(z^{\prime})=({\textrm{Supp} \ T}\cap L_g)\cap (\lbrace z^{\prime}\rbrace\times \bigtriangleup^{\prime\prime})$. Since $\mathcal{H}_{2p-2}({L_{g }\cap \textrm{Supp} \ T})$ is locally finite, then by \cite{Sh} the set $L_{g}(z^{\prime})$ is a discrete subset for a.e. $z^{\prime}$. Without loss of generality, we may assume that $L_{g}(z^{\prime})$ is reduced to a single point $(z^{\prime},0)$. On the other hand, $T$ is  $\complex$-flat on $\Omega$. Thus, The slice $\langle T,\pi,z^{\prime}\rangle$ exists for a.e. $z^{\prime}$, and is a positive plurisubharmonic current of bi-dimension $(1,1)$ on $\Omega$, supported in $\lbrace z^{\prime}\rbrace\times \bigtriangleup^{n-p+1}$. Now, by Theorem \ref{single}, the sequence $\langle g_{j} T, \pi, z^{\prime}\rangle$ is weakly$^{*}$ convergent since $L_{g}(z^{\prime})$ is a single point. Hence the slice formula implies that
\begin{equation}
\begin{split}
\lim_{j \to \infty}\int_{\bigtriangleup^{\prime}\times \bigtriangleup^{\prime\prime}}  g_{j} T\wedge \pi^{*}\beta^{\prime p-1} \wedge \beta^{''} &=\lim_{j \to \infty} \int_{z^{\prime}} \langle g_{j} T, \pi,z^{\prime}\rangle \beta^{\prime p-1}\wedge \beta^{''}\\&=\int_{z^{\prime}} \langle  g T, \pi,z^{\prime}\rangle \beta^{\prime p-1}\wedge \beta^{''}\\&= \int_{\bigtriangleup^{\prime}\times \bigtriangleup^{\prime\prime}} g T\wedge \pi^{*}\beta^{\prime p-1}\wedge \beta^{''}.
\end{split}
\end{equation}
And our desired current is achieved.
\end{proof}
\section{The Current $\ddc g \wedge T$}

As mentioned earlier in the introduction, for the case under investigation the current $\ddc g \wedge T $ stole the show from the current $gT$. Actually, Alessandrini-Bassanilli  \cite{Al-Ba2} and Al Abdulaali \cite{Ah4} studied the definition of $\ddc g \wedge T$ for pluriharmonic current $T$ and $g$ of class $\mathcal{C}^{2}$ apart of its locus points. Al Abdulaali  \cite{Ah6} generalized the latter works to the more general case when $T$ is plurisubharmonic and $g$ of class $\mathcal{C}^{1}$ where $\mathcal{H}_{2p-2}(L_{g})$ is locally finite. In \cite{Di-Si1}, Dihn and Sibony discussed the case when $\Omega$ is a compact  K\"{a}hler manifold. They obtained the desired current when $T$ is pluriharmonic and $g$ is continuous on $\Omega$. 
   
\begin{thm}\label{2nd:thm}
Under the same hypotheses of Theorem \ref{1st:main}, the current $\ddc g \wedge T$ is well defined.
\end{thm}

\begin{proof}
For any $\varphi \in \mathcal{C}^{\infty}_{0}(\Omega)$ we have
\begin{equation}
\int_{\Omega} \varphi \ddc (g_{j}T) \wedge \beta^{p-1}=\int_{\Omega} g_{j}\ddc \varphi \wedge T \wedge \beta^{p-1}
\end{equation}
This means that 
\begin{equation}
\begin{split}
\int_{\Omega} \varphi \ddc g_{j}\wedge T \wedge \beta^{p-1}&=\int_{\Omega} g_{j}\ddc \varphi  \wedge T \wedge \beta^{p-1} 
\\ & \ \ -\int_{\Omega} \varphi  g_{j} \ddc T \wedge \beta^{p-1}\\
& \ \ -2 \int_{\Omega} \varphi \dd g_{j}  \wedge \dd^{c} T \wedge \beta^{p-1}
\end{split}
\end{equation}
In virtue of the previous results, the sequence $\ddc g_{j} \wedge T$ converges to a current denoted by $\ddc g \wedge T$.
\end{proof}

By analogous discussion as in the proof of Theorem \ref{A:H2p-1}, one can apply the slice formula to imply the following assertion. 

\begin{thm}
Under the same hypotheses of Theorem \ref{A:H2p-1}, the current $\ddc g \wedge T$ is well defined.
\end{thm}

Theorem \ref{2nd:thm} allows us to define the number $$\mu (T,g):=\displaystyle{\lim_{r \to -\infty}\int_{\{g<r\}}T \wedge \ddc g \wedge \beta^{p-1}}.$$
For such number one can obtain the following comparison result.

\begin{thm}
In addition to the hypotheses of Theorem \ref{1st:main}, if $\ddc T=0$ and $u \in Psh^{-}(\Omega)\cap \mathcal{C}^{1}(\Omega \setminus L_{u})$ such that $u \dd^{c}T$ exists and
\begin{equation}
l:= \lim \sup \frac{u(z)}{g(z)} < \infty \ \textrm{as} \ g(z)\rightarrow -\infty,
\end{equation}
then $\mu (T,u) \leq l \mu (T,g)$.
\end{thm}

\begin{proof}
We follow a similar technique as in \cite{De2}. Since $\lambda \mu(T,g)=\mu(T,\lambda g)$ for all $\lambda \geq 0$, it is enough to show the result for $l=1$. Set $u_{c}=\max_{\varepsilon}(u-c,g)$  where $c$ is a positive constant and $\max_{\varepsilon}(x_{1},x_{2})$ is by definition
\begin{equation}
\max_{\varepsilon}(x_{1},x_{2})=\max(x_{1},x_{2})\ast \alpha_{\varepsilon},
\end{equation}
 where $\alpha_{\varepsilon}$ is a regularization kernel on $\real^{2}$ depending only on $\Vert (x_{1},x_{2})\Vert$. Now take $r^{'}<b<r<0$. Notice that for $c$ large enough we have $u_{c}=g$ on $\{ b<g \leq r\}$. Therefore,
\begin{equation}
\int_{\{g<r\}}T \wedge \ddc (u_{c}-g)\wedge \beta^{p-1}= \int_{\{g<r\}} (u_{c}-g) \ddc T \wedge \beta^{p-1}=0.
\end{equation}
But the properties of $u$ imply that $\mu (T,u_{c})=\mu (T,u-c)=\mu (T,u)$.
Hence, 
\begin{equation}
\begin{split}
\int_{\{g<r\}} T \wedge \ddc g \wedge \beta^{p-1}&=\int_{\{ g<r\}}T \wedge \ddc u_{c} \wedge \beta^{p-1}\\
& \geq \int_{\{g<r^{'}\}} T \wedge \ddc u_{c}\wedge \beta^{p-1} \\
& \geq \int_{\{u_{c}<r^{'}\}}T \wedge \ddc u_{c}\wedge \beta^{p-1}.
\end{split}
\end{equation}
We finish the proof by letting first $r^{'} \to -\infty$, and secondly $r \to -\infty$. 
\end{proof}

In Theorem \ref{2nd:thm}, if we consider $T$ to be positive plurisuperharmonic, then the statement fails to remain true. The next example illustrates this fact. Notice that, based on \cite{Ah6}, such wedge product exists when the obstacle is assumed to be of zero $(2p-2)$-Hausdorff measure.

\begin{exmp}
In $\complex$, set $T=g=\log|z|^{2}$. Then $T$ is negative and plurisubharmonic on $\{|z|<1\}$ where $g\dd^{c} T$ is well defined. But despite the fact that $L_{g}=\{ 0\}$ is of locally finite $0$-Hausdorff measure, the mass of $\ddc g \wedge T$ explodes across $\{ 0\}$.
\end{exmp}

However, local potential currents can be very useful to our settings. Remember that, by \cite{Ben-El1}, if $T$ is positive and closed, then locally there exist a negative plurisubharmonic current $U$ of bi-dimension $(p+1,p+1)$ and a smooth form $R$ such that $T=\ddc U+R$. The current $U$ is called the local potential of $T$.  
\begin{cor}
Let $T$ be a positive plurisubharmonic current of bi-dimension $(p,p), \ p \geq 1$ on  $\Omega $ and $g \in Psh^{-}(\Omega)\cap \mathcal{C}^{1}(\Omega\setminus L_{g})$. If $L_{g}$ is a single point, then $ \ddc g\wedge S$ is a well defined current on $\Omega$ where $S$ is the potential of $\ddc T$.
\end{cor}
\begin{proof}
As our problem is local, one can assume that $\ddc T= \ddc S$. Now, if we set $F=T-S$, then we get a positive pluriharmonic current. Hence by Theorem \ref{single}, both currents $\ddc g \wedge F$ and $\ddc g \wedge T$ are well defined. Therefore, one can define $\ddc g \wedge S$ by $\ddc g \wedge T- \ddc g \wedge F$.
\end{proof}

We end this paper by showing a case where the $\ddc g \wedge T$ can be defined without paying any attention to the derivatives of $g$. 
\begin{cor}
Let $T$ be a positive or negative plurisubharmonic current of bi-dimension $(p,p), \ p \geq 1$ on  $\Omega $ and $g \in Psh^{-}(\Omega)\cap L^{\infty}_{loc}(\Omega)$. If $\dd T$ is of order zero, then $ \ddc g\wedge T$ is a well defined current on $\Omega$.
\end{cor}

\begin{proof}
It is so obvious that the currents $gT$, $g \dd^{c}T$ and $g \ddc T$ are well defined. Therefore, one can define
\begin{equation}
\ddc g \wedge T=\ddc (gT)-2 \dd g \wedge \dd^{c}T-g \ddc T.
\end{equation}
\end{proof}

\paragraph{\textbf{Data Availability Statement.}} Data sharing not applicable to this article as no datasets were generated or analysed during the current study. 
\paragraph{\textbf{Acknowledgment.}}
All gratitude to Professors Hassine El Mir and Noureddine Ghiloufi for the helpful conversations regrading the work. 

\paragraph{\textbf{Funding.}}
The Author acknowledges the Deanship of Scientific Research at King Faisal University for the financial support under Nasher Track (Grant No. 186223).

\end{document}